%% file: paper.tex
\documentclass[11pt]{article}

\usepackage{amsmath,amssymb,amsfonts,textcomp,amsthm,xifthen,psfrag,graphicx,color,MnSymbol}
\usepackage[T1]{fontenc}
\usepackage[final]{hyperref}
\usepackage{placeins}

\usepackage{pgfplots}
\pgfplotsset{compat=newest} 
\usepgfplotslibrary{colorbrewer}
\pgfplotsset{compat=newest}
\usepgfplotslibrary{groupplots}
\usepgfplotslibrary{external}
\usepackage[ruled,linesnumbered]{algorithm2e}

\definecolor{navyblue}{rgb}{0.0, 0.0, 0.5}

\newcommand{\abs}[1]{\left|#1\right|}
\newcommand{\norm}[1]{\left\|#1\right\|}
\newcommand{\bN}{\boldsymbol{N}}

\date{}

\input{header}

\title{Tensor finite elements for smectic liquid crystals%
\thanks{Supported by ANID-Chile through FONDECYT projects 1230013, 1250070}}
\author{
Thomas~F\"uhrer$^\dagger$
\and
Norbert Heuer\thanks{
Facultad de Matem\'aticas, Pontificia Universidad Cat\'olica de Chile,
Avenida Vicu\~na Mackenna 4860, Santiago, Chile,
email: {\tt \{thfuhrer,nheuer\}@uc.cl}}
\and
Torsten Lin\ss\thanks{
Fakult\"at f\"ur Mathematik und Informatik,
FernUniversit\"at in Hagen, Universit\"atsstra{\ss}e 11, 58097 Hagen, Germany,
email: {\tt torsten.linss@fernuni-hagen.de}}}

\begin{document}
\maketitle
\begin{abstract}

We present a tensor-based finite element scheme for a smectic-A liquid crystal model.
We propose a simple C\'ea-type finite element projection in the linear case
and prove its quasi-optimal convergence. Special emphasis is put on the formulation
and treatment of appropriate boundary conditions.
For the nonlinear case we present a formulation in two space dimensions and
prove the existence of a solution. We propose a discretization that extends
the linear case in Uzawa-fashion to the nonlinear case
by an additional Poisson solver. Numerical results illustrate the performance
and convergence of our schemes.


\end{abstract}

\section{Introduction} \label{sec_intro}

Liquid crystals are materials that have properties of isotropic liquids and
anisotropic crystals. The crystalline structure imposes a certain orientation and periodicity.
For an overview we refer to the classical text \cite{DeGennesP_93},
see also \cite{WangZZ_21_MCL} for a recent treatise that includes numerical methods.
Models distinguish between specific phases.
Nematic liquid crystals tend to align along a global direction given by a director field.
Their numerical analysis is relatively well understood for different model variants, see, e.g.,
\cite{AdlerAEML_15_EMF,BorthagarayNW_20_SPF,GuillenGGS_13_LMF,LiuW_00_ALC,LiuW_02_MMA,MaityMN_21_DGF,MaityMN_23_PPE,NochettoRY_22_GCP}.

In this paper we present and analyze finite element approximations for a smectic-A
liquid crystal model.
The smectic phase is one that allows for layers with director fields locally
aligned within layers. The smectic-A phase refers to director fields that are locally
perpendicular to the layers. Numerical analysis for smectic liquid crystals is scarce.
Guill\'en-Gonz\'alez and Tierra \cite{GuillenGT_15_ASA} analyzed
a finite difference/finite element approach for such a model. See also
Chen, Yang and Zhang \cite{ChenYZ_17_SOL} for subsequent adaptations.
Different finite element approaches have been recently studied by
Xia and Farrell \cite{XiaF_23_VNA} and
Farrell, Hamdan, and MacLachlan \cite{FarrellHM_FED}.
Their model is related to what we study here.

We consider a model for smectic-A liquid crystals proposed
by Pevnyi, Selinger, and Sluckin \cite{PevnyiSS_14_MSL}, in the following referred to
as PSS-model. Its solution minimizes
an energy of the form
\begin{align*}
   \int_\Omega f_s(u)
   + \frac B2 \int_\Omega \abs{\Grad\grad u + q^2\bnu\bnu^\top u}^2
   + \int_\Omega f_n(\bnu)
\end{align*}
with respective bulk and nematic energy densities
\[
   f_s(u)= \frac {m_1}2 u^2 + \frac {m_2}3 u^3 + \frac {m_3}4 u^4
   \quad\text{and}\quad
   f_n(\bnu) = \frac K2 |\Grad\bnu|^2.
\]
Here, $u$ is the density variation, $\bnu$ is the unit-length director field,
and $m_1$, $m_2$, $m_3$, $B$, $q$, $K$ are real constants.
Number $K$ is also referred to as the Frank constant and $q$ is the smectic wave number.
The density tends to form periodic structures with wavelength $2\pi/q$.

The PSS-model is motivated by several numerical difficulties caused by other models,
see the discussion in \cite{PevnyiSS_14_MSL}.
The smectic tensor $\bM:=\Grad\grad u+q^2\bnu\bnu^\top u$ represents a fourth-order term
that couples the bulk and nematic energies with scaling provided by constant $B$.
Though, the very appearance of the smectic tensor complicates the
use of finite elements to approximate $u\in H^2(\Omega)$.
Such regularity requires elements with high polynomial degrees like the Argyris-element,
or HCT-elements with a macro structure that are difficult to implement.
Alternatively, canonical strategies to avoid $H^2$-regular elements are, e.g.,
discontinuous Galerkin and interior penalty methods, mixed methods,
or other non-conforming approaches. The recent paper by Farrell, Hamdan, and MacLachlan
\cite{FarrellHM_FED} investigates such strategies for the linear case
with given director field $\bnu$ and only quadratic bulk energy density.
They particularly consider different types of boundary conditions, often only treated incidentally.

In this paper we re-consider the reduced PSS-model with quadratic bulk energy density,
and study both, the linear case with given director field and the nonlinear case
in two dimensions with unknown director field. We represent the director field by its
angle and ensure in this way the unit-length constraint.
This strategy is also applicable in three space dimensions.
Though it renders more complicated in that case because the
then spherical representation of the director field turns nonlinear.
Our underlying approximation is based on a conforming finite element approximation
of the smectic tensor. Its energy space is $H(\dDiv,\Omega;\SS)$ that consists of
symmetric $L_2$-regular tensors $\bM$ whose iterated divergence $\div\Div\bM$ is $L_2$-regular
(details are given below).
In principle this leads to a mixed finite element method for a fourth-order problem,
and all the known tools can be applied. Though, in the linear case $u$ can be
approximated as a postprocessing step once an approximation of $\bM$ is known.
In the nonlinear case we propose an Uzawa-type algorithm that also separates the
problem for $u$. In both cases our algorithm is a simple C\'ea finite element projection
of the smectic tensor onto a conforming subspace of $H(\dDiv,\Omega;\SS)$,
combined with a Poisson solver for the director field in the nonlinear case.
Our discretizations therefore lead to simple systems with symmetric positive definite matrices.
Additionally, our formulation is entirely based on energy spaces.
Consequently, appearing trace operators give rise to uniformly well-posed boundary terms
and thus reveal boundary conditions that are appropriate for the model problem under consideration.
They divert slightly but critically from the ones analyzed in \cite{FarrellHM_FED}.
This becomes particularly relevant in the case of non-convex structures with
incoming corners or edges where the regularity of tensors is reduced, cf.~\cite{FuehrerHN_19_UFK}.
Whether this is relevant in liquid crystal applications is unclear.
The construction of conforming approximations of $H(\dDiv,\Omega;\SS)$-tensors
and uniformly continuous traces is based on what we have learned from
Kirchhoff--Love and Reissner--Mindlin plate bending
\cite{FuehrerH_25_MKL,FuehrerHN_19_UFK,FuehrerHS_20_UFR}.

Two important issues remain open in this paper.
The bulk energy density of the general PSS-model has cubic and quartic contributions.
Their presence complicates the treatment of the density in this paper,
which then becomes tied to the representation of tensor $\bM$.
Second, the nematic energy density of the PSS-model requires $H^1$-regularity
of the director field and thus excludes relevant cases with singularities that stem from
discontinuous changes of $\bnu$. This motivated further model adjustments, see the papers by
Ball and Bedford \cite{BallB_15_DOP} and
Xia, MacLachlan, Atherton, and Farrell \cite{XiaMAF_21_SLG},
and the numerical approach by Xia and Farrell \cite{XiaF_23_VNA}.

An overview of the remainder is as follows. In the following section we present
the linear model problem in detail and formulate possible boundary conditions.
Section~\ref{sec_spaces} introduces energy spaces of the model problem and
formulates trace operators that are required for the well-posed representation of
boundary conditions. In Section~\ref{sec_VF} we develop a variational formulation
of the linear model problem and prove its well-posedness (Theorem~\ref{thm_VF}).
Section~\ref{sec_nonl} is devoted to the formulation of the nonlinear model problem
and concludes with the existence of a solution (Theorem~\ref{thm_VF_nonl}).
Discretizations of both problems are presented in Section~\ref{sec_disc}.
We show convergence and appropriate convergence orders for the linear case
(Theorem~\ref{thm_disc_lin}) and present an Uzawa-algorithm for a solution
of the discretized nonlinear problem in \S\ref{sec_disc_nlin}.
Several numerical experiments in two dimensions are reported in Section~\ref{sec_numeric}.
They confirm expected convergence properties for the linear and nonlinear models,
and also illustrate applicability of the method for director fields with jumps.

\section{Linear model problem} \label{sec_model}

We consider a simply-connected, bounded, polygonal/polyhedral Lipschitz domain $\Omega\subset\R^d$
($d=2,3$) and aim at minimizing the energy
\begin{align} \label{energy}
   E(u)= \frac B2 \int_\Omega \abs{\Grad\grad u + q^2\bT u}^2
                           + \frac m2 \int_\Omega u^2 - L(u)
\end{align}
for a certain functional $L$ and subject to appropriate boundary conditions.
Here, $\grad$ is the gradient operator, $\Grad$ is the symmetric gradient
so that $\Grad\grad u$ is Hessian of $u$, and $m,q>0$ and $0<B\le 1$ are constants,
whereas $\bT\in L_\infty(\Omega;\SS)$ with $\SS:=\{\bM\in\R^{d\times d};\; \bM=\bM^\top\}$.
A special case is $\bT=\bnu\bnu^\top$ (dyadic product)
with $\bnu\in L_2(\Omega;\R^d)$ and $|\bnu|^2=\bnu^\top\bnu=1$.
Functional $L$ contains a possible force term $\int_\Omega fu$ and may include
a boundary functional to represent non-homogeneous natural boundary conditions.
Throughout we assume that $f\in L_2(\Omega)$.

The formulation of boundary conditions requires some preparation. In the following,
$\Gamma=\partial\Omega$ is the boundary of $\Omega$.
For a sufficiently smooth tensor $\bQ$ with values in $\SS$ we define
\begin{equation} \label{pit}
   \pit(\bQ\bn):=\begin{cases}
      \bt\cdot\bQ\bn             & (d=2)\\
      \bQ\bn-(\bn\cdot\bQ\bn)\bn & (d=3)
   \end{cases}
   \quad\text{on $\Gamma$ (almost everywhere)}.
\end{equation}
In two dimensions, $\bt$ is the unit tangential vector along $\Gamma$ in positive orientation,
and in two and three dimensions, $\bn$ denotes the exterior unit normal vector along $\Gamma$.
Later, we use the same notation $\pit$, $\bn$, and $\bt$ on the boundary of elements
(edges in two dimensions and faces in three dimensions) where $\bn$ is directed towards the
exterior of the corresponding element and $\bt$ has local positive orientation.

In the following, $\Div$ is the row-wise divergence operator for tensors
(the adjoint of $-\Grad$), and $\div_\Gamma$ is the surface divergence on $\Gamma$
(adjoint of the negative surface gradient). We define
\begin{align} \label{nDiveff}
   \nDiv(\bQ):= \bigl(\bn\cdot\Div\bQ+\div_\Gamma\pit(\bQ\bn)\bigr)|_\Gamma.
\end{align}
Later we use the same notation on the boundary of elements instead of $\Gamma$.
In the mechanics of plate bending with bending tensor $\bQ$, this term is known as the
effective shear force. In two dimensions,
$\div_\Gamma\pit(\bQ\bn)=\partial_\bt(\bt\cdot\bQ\bn)$
with tangential derivative operator $\partial_\bt=\bt\cdot\grad$.

Additionally we need jumps of tangential-normal traces of tensors $\bQ$.
In two dimensions, let $\gamma\subset\Gamma$ be a vertex-singleton of the polygon
$\Gamma$ where two edges $F_1$, $F_2$ meet (it can be at an angle of $\pi$)
and, in three dimensions,
let $\gamma\subset\Gamma$ be the relatively open edge shared by two faces $F_1$, $F_2$.
In two dimensions the respective tangential vectors are $\bt_1$, $\bt_2$
and we assume that edge $F_2$ follows edge $F_1$ in positive orientation of $\Gamma$.
In three dimensions we denote by $\bn_{\partial F_1}$ and $\bn_{\partial F_2}$
the respective unit tangential vectors of the faces along their boundaries
that are exterior normals of the respective face boundary. Then
\begin{equation} \label{jump}
   [\pit(\bQ\bn)]_{\Omega,\gamma}
   :=
   \begin{cases}
     \bigl(\bt_2\cdot(\bQ\bn)|_{F_2} - \bt_1\cdot(\bQ\bn)|_{F_1}\bigr)\big|_\gamma & (d=2),\\
     \bigl(\bn_{\partial F_2}\cdot(\bQ\bn)|_{F_2}
   + \bn_{\partial F_1}\cdot(\bQ\bn)|_{F_1}\bigr)\big|_\gamma & (d=3)
   \end{cases}
\end{equation}
is the tangential-normal jump of $\bQ$ across $\gamma$.
Below, when considering discretizations, the analogous notation
$[\pit(\bQ\bn)]_{\el,\gamma}$ will be used on the boundary of elements $\el$
(as traces from within~$\el$) where $\gamma$ is a vertex-singleton of $\el$ in two dimensions
or an edge in three dimensions.

To consider general boundary conditions we decompose $\Gamma$
into non-intersecting relatively open pieces
and two sets $\JD$, $\JN$ of disjoint vertex-singletons ($d=2$)
or edges ($d=3$) which also do not intersect with the other sets,
\begin{align*}  
   \Gamma=\mathrm{closure}\bigl(\G{hc}\cup\G{ss}\cup\G{sc}\cup\G{f}
                                \cup (\bigcup_{\gamma\in J} \gamma)\bigr)
   \quad\text{with}\ J:=\JD\cup\JN.
\end{align*}
The sets $\G{hc},\G{ss},\G{sc},\G{f}$ can be empty or are unions of connected sets
of positive relative measure. For our discretization we will have to assume that
they are unions of straight lines ($d=2$) or surface triangles ($d=3$).
The set $J$ can be empty and we require that every $\gamma\in J$
has positive distance to $\G{hc}\cup\G{ss}$.

Using the terminology from mechanics, we assign boundary conditions of hard clamped
(Dirichlet), simply supported (mixed), soft clamped (mixed),
and free (Neumann) types on the respective pieces, and prescribe
point/trace values at $\gamma\in\JD$ and jump conditions at $\gamma\in \JN$.
To simplify the presentation, these conditions are represented
by corresponding extensions $g$ (scalar function)
and $\bG$ (symmetric tensor function) onto $\Omega$ of given boundary data.
With the abbreviation
\[
   \bM:=\Grad\grad u + q^2\bT u
\]
the boundary conditions read
\begin{subequations} \label{bc}
\begin{align}
   &u=g,\quad \bn\cdot\grad u=\bn\cdot\grad g &&\hspace{-5em}\text{on}\ \G{hc},\\
   &u=g,\quad \bn\cdot\bM\bn=\bn\cdot\bG\bn   &&\hspace{-5em}\text{on}\ \G{ss},\\
   &\bn\cdot\grad u=\bn\cdot\grad g,\quad \nDiv(\bM)=\nDiv(\bG)
                                               &&\hspace{-5em}\text{on}\ \G{sc},\\
   &\nDiv(\bM)=\nDiv(\bG),\quad \bn\cdot\bM\bn = \bn\cdot\bG\bn &&\hspace{-5em}\text{on}\ \G{f},\\
   &u|_\gamma=g|_\gamma\quad\forall\gamma\in\JD,\quad
   [\pit(\bM\bn)]_{\Omega,\gamma}=[\pit(\bG\bn)]_{\Omega,\gamma}\quad \forall\gamma\in\JN.
\end{align}
\end{subequations}
Such boundary conditions appear naturally when integrating by parts the energy terms.
Their well-posedness requires the use of appropriate Banach spaces and corresponding
trace operators, introduced next.

\section{Spaces and trace operators} \label{sec_spaces}

We introduce spaces and trace operators needed for a variational formulation
of our model problem. Our setup has been developed in connection with the
study of ultraweak formulations for plate problems, cf.~\cite{FuehrerHN_19_UFK}.

For $\omega\subset\Omega$ we use the standard Sobolev spaces
$L_2(\omega)$, $H^2(\omega)$, and denote the corresponding spaces of
$U$-valued functions as $L_2(\omega;U)$, $H^2(\omega;U)$ for $U\in\{\R,\R^d,\SS\}$.
We equip $\SS$ with the Frobenius inner product
\begin{gather*}
  \bM\colon\bN  := \sum_{i,j=1}^d m_{ij} n_{ij}
  \quad \bigl(\bM=(m_{ij}), \bN=(n_{ij})\bigr),
\end{gather*}
and induced (squared) norm $\abs{\bM}^2 := \bM:\bM$.
The standard $L_2(\omega,U)$ inner products (extension by duality) and norms are generically
denoted as $\vdual{\cdot}{\cdot}_\omega$ and $\|\cdot\|_\omega$.
We drop the index $\omega$ if $\omega=\Omega$. For surfaces and lines $\omega$,
the $L_2(\omega)$ duality is $\dual{\cdot}{\cdot}_\omega$.
To specify approximation orders we also need Sobolev spaces $H^s(\Omega)$
of fractional order $s\in (0,2]$ with norm $\|\cdot\|_s$. For details we refer to~\cite{Adams}.

The space $H^2(\Omega)$ is furnished with the (squared) norm
\[
   \|v\|_\Ggrad^2 := m \|v\|^2 + B \|\Gradgrad v+q^2\bT v\|^2.
\]
We define the tensor space
\[
   H(\dDiv,\Omega;\SS) := \{\bM\in L_2(\Omega;\SS);\; \divDiv\bM\in L_2(\Omega)\}
\]
with inner product
\begin{gather*}
  \ip{\bM}{\bN}_\dDiv
    := B \vdual{\bM}{\bN}
       + \frac{B^2}{m} \vdual{\divDiv\bM + q^2\bT\colon\bM}
                           {\divDiv\bN + q^2\bT\colon\bN}
\end{gather*}
and induced (squared) norm
\[
   \norm{\bM}_{\dDiv}^2
     := \ip{\bM}{\bM}_\dDiv
     =  B \norm{\bM}^2 + \frac {B^2}m \norm{\divDiv\bM + q^2\bT\colon\bM}^2.
\]
As shown in \cite{FuehrerHN_19_UFK} for canonical norms,
traces of $H(\dDiv,\Omega;\SS)$ and $H^2(\Omega)$ are in duality.
Specifically, we define the trace operators
\begin{subequations}
  \label{def:trGg+trdD}
\begin{align}
   \label{def:trGg}
   \trGg{}:\; & \left\{\begin{array}{cll}
      H^2(\Omega) & \to & H(\dDiv,\Omega;\SS)^*\\
      v & \mapsto & \dual{\trGg{}(v)}{\bM} := \vdual{\divDiv\bM}{v} - \vdual{\bM}{\Gradgrad v},
   \end{array}\right.\\
   \label{def:trdD}
   \trdD{}:\; & \left\{\begin{array}{cll}
      H(\dDiv,\Omega;\SS) & \to & H^2(\Omega)^*\\
      \bM & \mapsto & \dual{\trdD{}(\bM)}{v} := \vdual{\divDiv\bM}{v} - \vdual{\bM}{\Gradgrad v}
   \end{array}\right.
\end{align}
\end{subequations}
and spaces with homogeneous boundary conditions
\begin{align*}
   &H^2_D(\Omega):=\{v\in H^2(\Omega);\;
                     v|_{\G{hc}\cup\G{ss}}=0,\
                     \bn\cdot\grad v|_{\G{hc}\cup\G{sc}}=0,\
                     v|_\gamma=0\; \forall\gamma\in J_D\},\\
   &H_N(\dDiv,\Omega;\SS):=
   \{\bM\in H(\dDiv,\Omega;\SS);\; \dual{\trdD{}(\bM)}{v}=0\ \forall v\in H^2_D(\Omega)\}.
\end{align*}
The following lemma shows the needed duality properties.

\begin{lemma} \label{la_tr}
(i) Any $v\in H^2(\Omega)$ satisfies
\[
   v\in H^2_D(\Omega) \quad\Leftrightarrow\quad
   \dual{\trGg{}(v)}{\dbM} = 0\quad\forall \dbM\in H_N(\dDiv,\Omega;\SS).
\]
(ii)
Trace operator $\trdD{}$ is bounded as
\[
   B \dual{\trdD{}(\bM)}{v}\le \|\bM\|_\dDiv \|v\|_\Ggrad\quad
   \forall \bM\in H(\dDiv,\Omega;\SS),\ v\in H^2(\Omega).
\]
\end{lemma}

\begin{proof}
Direction ``$\Leftarrow$'' of statement (i) is immediate by considering Green's identity
for smooth tensors $\dbM$, cf.~the proof of \cite[Proposition~3.8(i)]{FuehrerHN_19_UFK}.
Direction ``$\Rightarrow$'' holds by definition of $H_N(\dDiv,\Omega;\SS)$ and the fact that
\begin{gather}
  \label{trvM=trMv}
  \dual{\trGg{}(v)}{\bM}=\dual{\trdD{}(\bM)}{v} \ \ \ \forall
    v\in H^2(\Omega), \ \forall \bM\in H(\dDiv,\Omega;\SS),
\end{gather}
which follows from~\eqref{def:trGg+trdD}.

Bound (ii) is due to the Cauchy--Schwarz inequality and the addition and subtraction of
an appropriate term,
\begin{align*}
   \dual{\trdD{}(\bM)}{v}
   &=
   \vdual{\divDiv\bM}{v}-\vdual{\bM}{\Gradgrad v}\\
   &=
   \vdual{\divDiv\bM+q^2\bT\colon\bM}{v}-\vdual{\bM}{\Gradgrad v+q^2\bT v}
   \le \frac 1B \|\bM\|_\dDiv \|v\|_\Ggrad.
\end{align*}
\end{proof}

\begin{remark}
By the selection of weights in $\|\cdot\|_\Ggrad$ and $\|\cdot\|_\dDiv$,
the duality between traces of $H^2(\Omega)$ and $H(\dDiv,\Omega;\SS)$ involves
the weight $B$ as shown in the second statement of Lemma~\ref{la_tr}.
Of course, this dependence on $B$ can be eliminated by rescaling $\|\cdot\|_\dDiv$ with $B^{-1}$.
We prefer the current weighting since, up to the factor $1/2$,
$\|\cdot\|_\Ggrad$ is the energy norm and the $L_2$-term of $\|\cdot\|_\dDiv$
is also part of the energy norm. This ensures that both norms are equilibrated
and lead to comparable error levels in our numerical experiments.
\end{remark}

\section{Variational formulation of the linear problem} \label{sec_VF}

Trace operators $\trGg{}$ and $\trdD{}$ induce dual pairings among different types of
boundary trace components. Indeed, for $v\in H^2(\Omega)$ and sufficiently smooth
$\bQ\in H(\dDiv,\Omega;\SS)$, integration by parts reveals that
\begin{align} \label{IP}
   &\dual{\trGg{}(v)}{\bQ} = \dual{\trdD{}(\bQ)}{v}
   = \vdual{\div\Div\bQ}{v} - \vdual{\bQ}{\Grad\grad v} \nonumber\\
   &\quad= \dual{\bn\cdot\Div\bQ+\div_\Gamma\pit(\bQ\bn)}{v}_{\Gamma^0}
   - \dual{\bn\cdot\bQ\bn}{\bn\cdot\grad v}_{\Gamma^0}
   + \sum_{\gamma\in \cE_\Gamma} \dual{[\pit(\bQ\bn)]_{\Omega,\gamma}}{v}_\gamma.
\end{align}
Here, $\Gamma^0:=\Gamma\setminus\bigcup\limits_{\gamma\in\cE_\Gamma}\gamma$
denotes the boundary without the vertices ($d=2$) resp. edges ($d=3$) of $\Omega$,
the latter forming the set $\cE_\Gamma$.
Furthermore, $\dual{\cdot}{\cdot}_{\Gamma^0}$ is the edge/face-piecewise $L_2$
bilinear form and its extension by duality.
Operator $\pit$ and jump $[\pit(\bQ\bn)]_{\Omega,\gamma}$ have been defined
in \eqref{pit} and \eqref{jump}, respectively, and
\[
   \dual{[\pit(\bQ\bn)]_{\Omega,\gamma}}{v}_\gamma := \begin{cases}
      [\pit(\bQ\bn)]_{\Omega,\gamma}\, v(x)\quad &(\gamma=\{x\}, d=2),\\
      \int_\gamma [\pit(\bQ\bn)]_{\Omega,\gamma}\, v\quad &(d=3).
   \end{cases}
\]
For details we refer to \cite{FuehrerHN_19_UFK}.
We introduce the tensor variable
\begin{gather}
  \label{M}
  \bM := \Gradgrad u + q^2 \bT u.
\end{gather}
Then, relation \eqref{IP} shows that boundary conditions \eqref{bc} can be
represented as
\begin{align}
   \dual{\trGg{}(u)}{\dbM} &= \dual{\trGg{}(g)}{\dbM}
   &&\hspace{-4em}\forall\dbM\in H_N(\dDiv,\Omega;\SS), \label{bc_u}\\
   \dual{\trdD{}(\bM)}{\du} &= \dual{\trdD{}(\bG)}{\du} \label{bc_M}
   &&\hspace{-4em}\forall\du\in H^2_D(\Omega).
\end{align}
The boundary conditions for $\bM$ are natural with respect to the energy functional
\eqref{energy}, and lead to the selection of the functional
\[
   L(u) = \vdual{f}{u} - B \dual{\trdD{}(\bG)}{u}
\]
with a volume force $f$. We assume that $f\in L_2(\Omega)$.
The resulting formulation of the linear liquid crystal model with boundary conditions
\eqref{bc} as a minimization problem reads
\begin{equation} \label{min}
   u=\argmin_{v\in V_g}
     \frac B2 \int_\Omega \abs{\Grad\grad v + q^2\bT v}^2
   + \frac m2 \int_\Omega v^2
   - \int_\Omega f v + B \dual{\trdD{}(\bG)}{v}
\end{equation}
where
\[
   V_g := \Bigl\{v\in H^2(\Omega); \dual{\trGg{}(v)}{\dbM} = \dual{\trGg{}(g)}{\dbM}\
                              \forall\dbM\in H_N(\dDiv,\Omega;\SS)\Bigr\}.
\]
Let us derive a variational formulation of \eqref{min}.
Making use of the tensor variable $\bM$, variation of $v$ by functions of
$H^2_0(\Omega)$ (with homogeneous Dirichlet data on $\Gamma$) shows that
\begin{align} \label{eq_u}
   B\divDiv (\Gradgrad u + q^2 \bT u) + B (\Gradgrad u+ q^2\bT u) \colon q^2\bT  + m u &=
   \nonumber\\
   B(\divDiv\bM + q^2\bT\colon\bM) + m u  &= f.
\end{align}
Boundary condition \eqref{bc_M} can be shown to hold by variation with functions of $H^2_D(\Omega)$.
Testing~\eqref{eq_u} with $\divDiv\dbM + q^2\bT\colon\dbM$ for $\dbM\in H_N(\dDiv,\Omega;\SS)$
gives
\begin{multline*}
   B \vdual{\divDiv\bM + q^2\bT\colon\bM}{\divDiv \dbM + q^2\bT\colon\dbM}
      + m \vdual{u}{\divDiv \dbM + q^2\bT\colon\dbM} \\
   = \vdual{f}{\divDiv \dbM + q^2\bT\colon\dbM},
\end{multline*}
and an application of trace operator $\trGg{}$ to relation \eqref{M} shows that
\begin{gather*}
  m \vdual{\bM}{\dbM}
    = m \vdual{u}{\divDiv \dbM + q^2\bT\colon\dbM} - m \dual{\trGg{}(u)}{\dbM}.
\end{gather*}
Adding the last two equations eliminates $u$, and  an inclusion of boundary conditions
\eqref{bc_u}, \eqref{bc_M} results in a single variational equation
for $\bM$ and explicit representation of $u$ through \eqref{eq_u}:
\emph{Find $\bM\in H(\dDiv,\Omega;\SS)$ such that}
\begin{subequations} \label{probM}
\begin{align}
   \label{probM1}
   & \ip{\bM}{\dbM}_\dDiv
   =  \frac Bm \vdual{f}{\divDiv\dbM + q^2\bT\colon\dbM}
   - B \dual{\trGg{}(g)}{\dbM},\\
   \label{probM2}
   &\dual{\trdD{}(\bM)}{\du} = \dual{\trdD{}(\bG)}{\du}\\
\intertext{\emph{for any $\dbM\in H_N(\dDiv,\Omega;\SS)$ and $\du\in H^2_D(\Omega)$, and set}}
   &u := \frac 1m f-\frac Bm(\divDiv\bM + q^2\bT\colon\bM). \label{u}
\end{align}
\end{subequations}
Note that $g\in H^2(\Omega)$ and $\bG\in H(\dDiv,\Omega;\SS)$
are arbitrary extensions of the boundary data. Therefore, norms of $g$ and $\bG$
can be considered to be infima, thus giving rise to trace norms of boundary data.
To keep things simple we do not specify these norms explicitly.
We show that problem \eqref{probM} is well posed and provides the solution of \eqref{min}.

\begin{theorem} \label{thm_VF}
Given $f\in L_2(\Omega)$, $\bG\in H(\dDiv,\Omega;\SS)$, and $g\in H^2(\Omega)$,
problem \eqref{probM1}, \eqref{probM2} has a unique solution $\bM$ with bound
\[
   \|\bM\|_\dDiv \le \frac 1{\sqrt{m}} \|f\| + 2 \|\bG\|_\dDiv + \|g\|_\Ggrad.
\]
Furthermore, $u$ defined by \eqref{u} solves \eqref{min},
and $\bM$ equals $\Grad\grad u+q^2\bT u$.
In particular, $u\in H^2(\Omega)$ satisfies \eqref{bc_u}, and is bounded as
\begin{align*}
   \norm{u}_\Ggrad \le 
      \frac{3}{\sqrt{m}} \|f\| + 3 \|\bG\|_\dDiv + 2 \|g\|_\Ggrad
\end{align*}
and
\[
   \|u\| \le \frac 1m \|f\| + \frac Bm \|\divDiv\bM+q^2\bT\colon\bM\|
   \le \frac 1m \|f\| + \frac 1{\sqrt{m}} \|\bM\|_\dDiv.
\]
\end{theorem}

\begin{proof}
Problem \eqref{probM1}, \eqref{probM2} is well posed as a projection,
noting the boundedness of the right-hand side and trace functionals.
Setting $\wtilde\bM:=\bM-\bG\in H_N(\dDiv,\Omega;\SS)$,
\eqref{probM1}, \eqref{trvM=trMv} and Lemma~\ref{la_tr}(ii) show that
\begin{align*}
   \|\wtilde\bM\|_\dDiv^2
   &\le \frac 1{\sqrt{m}} \|f\|
        \frac B{\sqrt{m}} \|\divDiv\wtilde\bM + q^2\bT\colon\wtilde\bM\|
   +   \bigl(\|\bG\|_\dDiv+\|g\|_\Ggrad\bigr) \|\wtilde\bM\|_\dDiv\\
   &\le \bigl(\frac 1{\sqrt{m}}\|f\| + \|\bG\|_\dDiv + \|g\|_\Ggrad\bigr)
          \norm{\wtilde\bM}_\dDiv.
\end{align*}
An application of the triangle inequality yields the claimed bound for $\|\bM\|_\dDiv$.

Next, we show that \mbox{$u\in H^2$} is bounded as stated and satisfies the boundary conditions.
By definition \eqref{u} and the regularities of $f$ and $\bM$, we have $u\in L_2(\Omega)$.
We calculate $\Gradgrad u+q^2\bT u$ in the distributional sense.
Given $\dbM\in C^\infty_0(\Omega;\SS)$, we find by definition of $u$ and relation
\eqref{probM1} that
\begin{align*}
  (\Gradgrad u+q^2\bT u)(\dbM)
    & = \vdual{u}{\divDiv\dbM + q^2\bT\colon\dbM}\\
    & = \vdual{\frac 1m f - \frac Bm (\divDiv\bM+q^2\bT\colon\bM)}{\divDiv\dbM + q^2\bT\colon\dbM}
      = \vdual{\bM}{\dbM},
\end{align*}
that is,
\begin{gather*}
   \Gradgrad u+q^2\bT u = \bM \in H(\dDiv,\Omega;\SS).
\end{gather*}
This implies \mbox{$u\in H^2(\Omega)$}.
Using this regularity, 
\eqref{u} and \mbox{$\bM=\Gradgrad u+q^2\bT u$},
an application of trace operator $\trGg{}$
and relation \eqref{probM1} show that
\begin{align*}
   \dual{\trGg{}(u)}{\dbM}
   = \vdual{u}{\divDiv\dbM+q^2\bT\colon\dbM}
   - \vdual{\Gradgrad u + q^2\bT u}{\dbM} 
   = \dual{\trGg{}(g)}{\dbM}
\end{align*}
for any $\dbM\in H_N(\dDiv,\Omega;\SS)$, thus confirming that $u$ satisfies boundary condition
\eqref{bc_u}. It is clear that $u$ is the unique solution of \eqref{min} with boundary conditions
\eqref{bc} formulated as \eqref{bc_u}, \eqref{bc_M}.
The bound for $\|u\|_\Ggrad$ is due to relations \eqref{u},
$\bM=\Gradgrad u+q^2\bT u$, and the bound on $\bM$,
\begin{align*}
  \|u\|_\Ggrad
     & \le \sqrt{m} \|u\| + \sqrt{B} \|\bM\|
       \le \frac 1{\sqrt{m}} \|f\|
                + \frac B{\sqrt{m}} \|\divDiv\bM+q^2\bT\colon\bM\|
                + \sqrt{B} \|\bM\|\\
     & \le \frac 1{\sqrt{m}}\|f\| + \sqrt{2} \|\bM\|_\dDiv
       \le \frac 1{\sqrt{m}} \|f\|
             + \sqrt{2} \left(\frac 1{\sqrt{m}} \|f\| + 2 \|\bG\|_\dDiv + \|g\|_\Ggrad
                          \right)\\
     & = \frac{1+\sqrt{2}}{\sqrt{m}} \|f\|
           + 2^{3/2} \|\bG\|_\dDiv + \sqrt{2} \|g\|_\Ggrad.
\end{align*}
This finishes the proof.
\end{proof}

\section{Nonlinear model in two dimensions} \label{sec_nonl}

Model \eqref{min} is a linear simplification of the nonlinear model with energy
\begin{align*}
   \frac B2 \int_\Omega \abs{\Grad\grad v + q^2\bnu\bnu^\top v}^2
   + \frac m2 \int_\Omega v^2
   + \frac K2 \int_\Omega |\Grad\bnu|^2 - L(v)
\end{align*}
where $K>0$ and also the tensor $\bT=\bnu\bnu^\top$ with $|\bnu|=1$ is unknown,
see~\cite{FarrellHM_FED,PevnyiSS_14_MSL}.
This model requires $\bnu\in H^1(\Omega;\R^d)$ which excludes
some relevant cases where $\bnu$ has point or line singularities.
Appropriate boundary conditions for $\bnu$ are essential and can include
periodicity. For simplicity we assume a given direction of $\bnu$ on the whole boundary.

We consider the two-dimensional case ($d=2)$ and
use the representation $\bnu=(\cos\phi,\sin\phi)^\top$. Then
\begin{align*}
   |\Grad\bnu|^2 = |(-\sin\phi,\cos\phi)^\top \grad\phi^\top|
   = |\grad\phi|^2,
\end{align*}
and, maintaining boundary conditions \eqref{bc} and imposing $\phi=\eta$ on $\Gamma$ for
$\eta\in H^1(\Omega)$, the nonlinear minimization problem becomes
\begin{subequations} \label{min_nonl}
\begin{align}
   &(u,\phi) = \argmin \{E(v,\psi);\; v\in V_g,\; \psi\in W_\eta\}
\end{align}
with
\begin{align}
   &E(v,\psi) =
   \frac B2 \int_\Omega \abs{\Grad\grad v + q^2\bT(\psi)v}^2
   + \frac m2 \int_\Omega v^2
   + \frac K2 \int_\Omega |\grad\psi|^2
   - \int_\Omega f v + B \dual{\trdD{}(\bG)}{v}.
\end{align}
\end{subequations}
Here,
\[
   \bT(\psi) = \begin{pmatrix} \cos^2\psi & \sin\psi\cos\psi\\
                               \sin\psi\cos\psi & \sin^2\psi
               \end{pmatrix}
\]
and, repeating the definition of $V_g$,
\begin{align*}
   V_g &:= \Bigl\{v\in H^2(\Omega); \dual{\trGg{}(v)}{\dbM} = \dual{\trGg{}(g)}{\dbM}\
                              \forall\dbM\in H_N(\dDiv,\Omega;\SS)\Bigr\},\\
   W_\eta &:= \Bigl\{\psi\in H^1(\Omega);\; \psi-\eta\in H^1_0(\Omega)\}.
\end{align*}
Analogously as before we introduce the tensor
\begin{equation*}
   \bM=\bM(u,\phi):=\Grad\grad u+q^2\bT(\phi)u
\end{equation*}
and obtain the Euler--Lagrange equations of \eqref{min_nonl},
\begin{align*} 
   B(\divDiv\bM + q^2\bT(\phi)\colon\bM) + m u &= f,\\
   -K \Delta\phi + Bq^2\bM\colon \bT'(\phi)u &= 0.
\end{align*}
Here,
\[
   \bT'(\phi) = \begin{pmatrix} -2\sin\phi\cos\phi & \cos^2\phi-\sin^2\phi\\
                                \cos^2\phi-\sin^2\phi & 2\sin\phi\cos\phi
                \end{pmatrix}
              = \begin{pmatrix} -\sin(2\phi) & \cos(2\phi)\\
                                \cos(2\phi) & \sin(2\phi)
                \end{pmatrix}
\]
is the derivative of $\bT(\phi)$ with respect to $\phi$.
In particular, the nonlinear energy leads to a $\phi$-dependent inner product
in $H(\dDiv,\Omega;\SS)$,
\[
   \ip{\bM}{\dbM}_{\dDiv,\phi}
   := B \vdual{\bM}{\dbM}
   + \frac{B^2}{m} \vdual{\divDiv\bM + q^2\bT(\phi)\colon\bM}
                           {\divDiv\dbM + q^2\bT(\phi)\colon\dbM}.
\]
We proceed with unknown $\bM$ as before, cf.~\eqref{probM}, keep considering $u$ as
an $L_2(\Omega)$ unknown that now depends on $\bM$ and $\phi$, and add a standard
variational formulation of the semilinear Poisson problem for $\phi$.
This leads to a nonlinear variational system for
$\bM\in H(\dDiv,\Omega;\SS)$, $u\in L_2(\Omega)$, and $\phi\in H^1(\Omega)$:
\begin{subequations} \label{Mphiu}
\begin{align}
   \label{Mphiu1}
   &\ip{\bM}{\dbM}_{\dDiv,\phi}
   =  \frac Bm \vdual{f}{\divDiv\dbM + q^2\bT(\phi)\colon\dbM}
   - B\dual{\trGg{}(g)}{\dbM},\\
   \label{Mphiu2}
   &\vdual{u}{\du}+\frac Bm \vdual{\divDiv\bM + q^2\bT(\phi)\colon\bM}{\du}
   = \frac 1m \vdual{f}{\du},\\
   \label{Mphiu3}
   &K \vdual{\grad\phi}{\grad\dphi} + Bq^2\vdual{\bM\colon \bT'(\phi)u}{\dphi}
   = 0,\\
   \label{Mphiu4}
   &\dual{\trdD{}(\bM)}{\dpsi} = \dual{\trdD{}(\bG)}{\dpsi}
   \qquad\text{and}\qquad 
   \phi=\eta\ \text{on}\ \Gamma
\end{align}
\end{subequations}
for any $\dbM\in H_N(\dDiv,\Omega;\SS)$, $\du\in L_2(\Omega)$, $\dphi\in H^1_0(\Omega)$,
$\dpsi\in H^2_D(\Omega)$.

This problem has a solution.

\begin{theorem} \label{thm_VF_nonl}
Given $B,K,m>0$, $f\in L_2(\Omega)$, $\bG\in H(\dDiv,\Omega;\SS)$, $g\in H^2(\Omega)$, and
$\eta\in H^1(\Omega)$, problem \eqref{min_nonl} has a solution $u,\phi$,
and so has \eqref{Mphiu} with $u\in H^2(\Omega)$ and $\bM=\Grad\grad u+q^2\bT(\phi)u$
satisfying boundary conditions \eqref{bc}.
\end{theorem}

\begin{proof}
Once having a solution $u,\phi$ to \eqref{min_nonl}
the rest of the statement follows.
To prove the existence of a solution to \eqref{min_nonl}
we use standard arguments from the calculus of variations.
By considering $u-g$ and $\phi-\eta$ instead of $u$ and $\phi$, respectively,
we can assume that we are minimizing over
$(v,\psi)\in V_0\times W_0:=H^2_D(\Omega)\times H^1_0(\Omega)$.
The existence follows from the coercivity and (sequential) weak lower semicontinuity
of $E$ in $V_0\times W_0$, see, e.g., \cite[Theorem~2.3]{Rindler_18_CV}.

By the lower triangle inequality, Young's inequality, the definition of
$\trdD{}(\bG)$, and relation $\|\bT(\psi)v\|^2 = \|v\|^2$, we bound with $\delta>1$
\begin{align*}
   E(v,\psi)
   &\ge
   \frac B2\Bigl((1-\delta^{-1})\|\Grad\grad v\|^2 + (1-\delta) \|q^2\bT(\psi)v\|^2\Bigr)
    + \frac m2 \|v\|^2
    + \frac K2 \|\grad\psi\|^2\\
    &\quad
    - \|f\| \|v\|
    - \frac B2 \bigl( \|\bG\|^2 + \|\div\Div\bG\|^2\bigr)^{1/2}
               \bigl(\|\Grad\grad v\|^2 + \|v\|^2\bigr)^{1/2}\\
    &=
   \frac B2 (1-\delta^{-1})\|\Grad\grad v\|^2 + \frac {m+(1-\delta)Bq^4}2 \|v\|^2 + \frac K2 \|\grad\psi\|^2\\
    &\quad
    - \|f\| \|v\|
    - \frac B2 \bigl( \|\bG\|^2 + \|\div\Div\bG\|^2\bigr)^{1/2}
               \bigl(\|\Grad\grad v\|^2 + \|v\|^2\bigr)^{1/2}.
\end{align*}
We select $\delta=1+\frac m{2Bq^4}$ so that $1-\delta^{-1}$ and $m+(1-\delta)Bq^4=m/2$ are positive,
and conclude that $E(v_k,\psi_k)\to\infty$ ($k\to\infty$) for any sequence
$((v_k,\psi_k))\subset V_0\times W_0$ with
$\|v_k\|+\|\Grad\grad v_k\| + \|\psi_k\| + \|\grad\psi_k\|\to\infty$.
Here, we also make use of the Poincar\'e inequality to control $\|\psi_k\|$ by $\|\grad\psi_k\|$.
This shows the coercivity of $E$.

To verify the weak lower semicontinuity of $E$ in $V_0\times W_0$ we note that
\begin{align*}
   &F:\;\Omega\times\R\times\R\times\R^{d\times d}\times\R^d \to [0,\infty),\\
   &F(x,(\xi_v,\xi_\psi),(\bA_v,A_\psi))
   := \frac B2 \abs{\bA_v + q^2\bT(\xi_\psi)\xi_v}^2
   + \frac m2 \xi_v^2 + \frac K2 |A_\psi|^2
\end{align*}
is a Carath\'eodory function because $x\mapsto F(x,(\xi_v,\xi_\psi),(\bA_v,A_\psi))$
is Lebesgue-measurable for every
$((\xi_v,\xi_\psi),(\bA_v,A_\psi))\in \R\times\R\times\R^{d\times d}\times\R^d$
and
$((\xi_v,\xi_\psi),(\bA_v,A_\psi))\mapsto F(x,(\xi_v,\xi_\psi),(\bA_v,A_\psi))$
is continuous almost everywhere on $\Omega$.
Furthermore, for almost every $x\in\Omega$,
$F(x,(\xi_v,\xi_\psi),(\cdot,\cdot))$ is convex for every $(\xi_v,\xi_\psi)\in\R\times\R$.
This shows that
\[
   E(v,\psi)=\int_\Omega F(x,(v,\psi),(\Grad\grad v,\grad\psi))\,\mathrm{d}x
   - \int_\Omega f v + B \dual{\trdD{}(\bG)}{v}
\]
is weakly lower semicontinuous in $V_0\times W_0$, see~\cite[Theorem~2.11]{Rindler_18_CV},
and finishes the proof.
\end{proof}

\section{Discretization} \label{sec_disc}

We start with a discretization of the linear problem \eqref{probM}
and describe our method for the nonlinear problem \eqref{Mphiu} in two dimensions
afterwards.

\subsection{Linear problem: projection} \label{sec_disc_lin}

We consider a shape-regular, regular decomposition $\mesh$ of $\Omega$ into open simplices $\el$.
By regular we mean that all elements are non-degenerate and two distinct but
touching elements either share one vertex or one edge or, if $d=3$, one face.
The set of relatively open faces (edges if $d=2$) of the mesh is denoted by $\cF$
and the set of relatively open edges (vertex singletons if $d=2$) by $\cE$.
We use $\cF(\el)$ and $\cE(\el)$ to denote the faces (edges) and edges (vertex singletons) of
an element $\el\in\mesh$. Furthermore, $\cF_0$ and $\cE_0$ denote the sets
of interior faces/edges and interior edges/vertex singletons ($d=3$/$d=2$),
respectively. We use the mesh-size function $h_\mesh$ with
$h_\mesh|_\el:=\diam(\el)$ for $\el\in\mesh$.

The space of scalar polynomials of degree $\le p$ on $\el\in\mesh$, resp. on $F\in\cF(\el)$,
is denoted by $P^p(\el)$, resp. $P^p(F)$, and the vector-valued polynomial space on $\el$
is $P^p(\el;\R^d)$. 
$P_\hom^p(\el)\subset P^p(\el)$ is the space of homogeneous polynomials of degree $p$.
The $L_2(\el)$ projection operator onto polynomial spaces of degree $p$ is denoted by
$\Pi^p_\el$, and the $\mesh$-piecewise operator by $\Pi^p_\mesh$.
We use the generic notation $\bn$ for exterior normal vectors on faces
$F\in\cF(\el)$.

The construction of our $H(\dDiv,\el;\SS)$-element is taken from \cite{FuehrerH_25_MKL}, with
extension to $d=3$ specified in \cite{ChenH_25_NDD}, and is based on the Raviart--Thomas space
\begin{align*}
  \RT^p(\el) = \bx P_\hom^p(\el) \oplus P^p(\el;\R^d),
  \quad \el\in\mesh.
\end{align*}
Here, we slightly abuse notation by writing $\bx$ for the function $\R^d\ni \bx \mapsto\bx$
and space $\bx P_\hom^p(\el) := \set{\bx\eta}{\eta\in P_\hom^p(\el)}$.
Similarly, $\bx^\top$ corresponds to the function $\bx\mapsto \bx^\top$.
On the other hand, points $x\in\R^d$ are written with standard italian font.
We will use dyadic products also of vector-valued functions, i.e.,
\begin{align*}
  \bphi\bpsi^\top:\; x\mapsto\bphi(x)\bpsi(x)^\top
\end{align*}
for $x$ in the (common) domain of vector functions $\bphi$, $\bpsi$, and denote
\begin{align*}
  X\otimes Y := \span\set{\bphi\bpsi^\top}{\bphi\in X,\ \bpsi\in Y}
\end{align*}
for spaces $X,Y$ of vector functions. Furthermore,
\begin{align*}
  \sym(\X) := \set{\sym(\bM)}{\bM\in\X}\quad\text{with}\quad
  \sym(\bM) := \frac12(\bM+\bM^\top)
\end{align*}
for tensor spaces $\X$.
Our local space of shape functions on an element $\el\in\mesh$ is defined as
\begin{align}
  \X(\el) = \sym(\RT^0(\el)\otimes\RT^1(\el))
\end{align}
with the following degrees of freedom,
\begin{subequations}\label{dofs}
\begin{alignat}{3}
    &\text{face moments}\qquad &&\dual{\bn\cdot\bM\bn}{q}_F
    &\quad  \bigl(q\in P^1(F),\ F\in\cF(\el)\bigr), \label{dofs1}\\
    &\text{face moments}\qquad
    &&\dual{\nDiv(\bM)}{q}_F
    &\quad \bigl(q\in P^1(F),\ F\in\cF(\el)\bigr), \label{dofs2}\\
    &\text{jump values}\qquad
    &&[\pit(\bM\bn)]_{\el,e}
    &\quad \bigl(e\in \cE(\el)\bigr)\quad\text{if $d=2$}, \label{dofs3}\\
    &\text{jump moments}\qquad
    &&\dual{[\pit(\bM\bn)]_{\el,e}}{q}_e
    &\quad \bigl(q\in P^1(e),\ e\in \cE(\el)\bigr)\quad\text{if $d=3$} \label{dofs4}
\end{alignat}
\end{subequations}
(recall definition \eqref{nDiveff} of operator $\nDiv$
and \eqref{pit},~\eqref{jump} for jump $[\pit(\bM\bn)]_{\el,e}$).

By \cite[Proposition~7]{FuehrerH_25_MKL}, \cite[\S 2.5]{ChenH_25_NDD}, the local spaces satisfy
\begin{align} \label{dDivlin}
   \divDiv\X(\el)=P^1(\el)\quad (\el\in\mesh).
\end{align}
We select the discrete spaces
\[
   \X(\mesh) := \bigl(\Pi_{\el\in\mesh} \X(\el)\bigr) \cap H(\dDiv,\Omega;\SS)
   \quad\text{and}\quad
   \X_N(\mesh) := \X(\mesh)\cap H_N(\dDiv,\Omega;\SS).
\]
Conformity is achieved by matching degrees of freedom \eqref{dofs1}, \eqref{dofs2}
for faces $F\in\cF_0$, and by requiring that the sums of individual degrees of freedom
\eqref{dofs3} resp. \eqref{dofs4} for individual edges (vertices) $e\in\cE_0$ vanish,
see \cite[Proposition~3.6]{FuehrerHN_19_UFK}.

Degrees of freedom \eqref{dofs} induce an interpolation operator
\[
   \Pi^\dDiv_\mesh:\; H(\dDiv,\Omega;\SS)\cap H^{1+r}(\mesh;\SS)\to \X(\mesh)
   \quad (r>1/2).
\]
Let us recall some of its properties.

\begin{lemma} \label{la_PidDiv}
Operator $\Pi^\dDiv_\mesh$ satisfies the commutativity property
\[
   \Pi^1_\mesh \circ \divDiv = \divDiv \circ \Pi^\dDiv_\mesh.
\]
Let $r>3/2$, $s>0$ and $\bM\in H^r(\Omega;\SS)$ be given with $\divDiv\bM\in H^s(\Omega)$.
The interpolation $\bM_h:=\Pi^\dDiv_\mesh\bM$ satisfies
\begin{align*}
   \|\bM-\bM_h\|
   &\le C \|h_\mesh\|_\infty^{\min\{r,2\}} \|\bM\|_{\min\{r,2\}},\\
   \|\divDiv(\bM-\bM_h)\|
   &\le C \|h_\mesh\|_\infty^{\min\{s,2\}} \|\divDiv\bM\|_{\min\{s,2\}}
\end{align*}
with a positive constant $C$ that is independent of
$\bM$ and $\mesh$ (under the shape-regularity assumption).
\end{lemma}

\begin{proof}
The error estimate for $\|\divDiv(\bM-\bM_h)\|$ is due to the commutativity
property. For $d=2$, this and the error estimate for $\|\bM-\bM_h\|$ are shown
by \cite[Proposition~10]{FuehrerH_25_MKL}. The proof for $d=3$ is analogous.
\end{proof}

Our discrete scheme reads as follows.
\emph{Find $\bM_h\in \X(\mesh)$ such that}
\begin{subequations} \label{probMh}
\begin{align}
   \dual{\trdD{}(\bM_h)}{\dpsi}&=\dual{\trdD{}(\Pi^\dDiv_\mesh\bG)}{\dpsi}
   \quad\forall\dpsi\in H^2_D(\Omega), \label{probMh_bc}\\
   \ip{\bM_h}{\dbM}_\dDiv
   &=  \frac Bm \vdual{f}{\divDiv\dbM + q^2\bT\colon\dbM}
   - B\dual{\trGg{}(g)}{\dbM}\quad\forall\dbM\in \X_N(\mesh).\\
\intertext{\emph{Then set}}
   \wtilde u &:= \frac 1m f-\frac Bm\Bigl(\divDiv\bM_h + q^2\bT\colon\bM_h\Bigr)
   \quad\text{\emph{and}}\quad
   u_h       := \Pi^1_\mesh \wtilde u.
\end{align}
\end{subequations}

\begin{remark} \label{rem_mixed}
(i) Relation \eqref{probMh_bc} uses test functions $\dpsi\in H^2_D(\Omega)$.
Therefore, boundary trace components of $\bM_h$ coincide with those of $\Pi^\dDiv_\mesh\bG$,
thus approximating the (now essential) boundary condition for $\bM$ in \eqref{Mphiu4}.\\
(ii)
In general $\wtilde u$ is not a piecewise polynomial due to the presence of $f$ and
$\bT$. Furthermore, $\bM_h\in P^3(\mesh;\SS)$ and $\divDiv\bM_h\in P^1(\mesh)$ by \eqref{dDivlin}.
The pair $(\bM_h,u_h)$ can be interpreted as the solution to a discrete mixed scheme of
\eqref{min} with approximation spaces $\X(\mesh)\subset H(\dDiv,\Omega;\SS)$ and
$P^1(\mesh)\subset L_2(\Omega)$.
\end{remark}

\begin{theorem} \label{thm_disc_lin}
The solutions $(\bM,u)$ and $(\bM_h,\widetilde u,u_h)$ of \eqref{probM} and \eqref{probMh},
respectively, satisfy
\[
   \sqrt{m} \|u-\wtilde u\| \le
   \|\bM-\bM_h\|_\dDiv \le \min \{\|\bM-\bQ\|_\dDiv;\; \bQ\in\X(\mesh)\}
\]
and
\[
   \sqrt{m} \|u-u_h\|
   \le \sqrt{m} \|u-\Pi^1_\mesh u\| + \min \{\|\bM-\bQ\|_\dDiv;\; \bQ\in\X(\mesh)\}.
\]
Assuming that $u\in H^s(\Omega)$, $\bM\in H^r(\Omega;\SS)$, and $\divDiv\bM\in H^t(\Omega)$
with $s,t\in (0,2]$ and \mbox{$r\in (3/2,2]$}, estimates
\begin{align*}
   \|\bM-\bM_h\|_\dDiv
   &\le C \|h_\mesh\|_\infty^{\min\{r,t\}}
       \Bigl(\sqrt{B}(1+\sqrt{\frac Bm}q^2)\|\bM\|_r + \frac{B}{\sqrt{m}} \|\divDiv\bM\|_t\Bigr),\\
   \|u-u_h\| &\le
   C \|h_\mesh\|_\infty^s \|u\|_s +
   C \|h_\mesh\|_\infty^{\min\{r,t\}}
   \Bigl(\sqrt{\frac Bm}(1+\sqrt{\frac Bm}q^2)\|\bM\|_r + \frac{B}{m} \|\divDiv\bM\|_t\Bigr)
\end{align*}
hold with a constant $C>0$ that is independent of $B$, $q$, $m$, and $\mesh$
(under the shape-regularity assumption).
\end{theorem}

\begin{proof}
The bound for $\|u-\wtilde u\|$ is due to relation
\[
   u-\wtilde u = \frac Bm \Bigl(\divDiv (\bM_h-\bM) + q^2\bT\colon (\bM_h-\bM)\Bigr)
\]
and the definition of $\|\cdot\|_\dDiv$.
The definition of $\bM_h$ as a projection implies the first bound for $\|\bM-\bM_h\|_\dDiv$.
We use \eqref{dDivlin} to calculate
\begin{align*}
   \sqrt{m} \|\Pi^1_\mesh u-u_h\|
   &\le \frac B{\sqrt{m}}\|\Pi^1_\mesh(\divDiv\bM+q^2\bT\colon\bM-\divDiv\bM_h-q^2\bT\colon\bM_h)\|
   \le \|\bM-\bM_h\|_\dDiv.
\end{align*}
An application of the triangle inequality yields
\begin{align*}
   \sqrt{m} \|u-u_h\| &\le \sqrt{m} \|u-\Pi^1_\mesh u\| + \|\bM-\bM_h\|_\dDiv.
\end{align*}
This proves the first bound for $\|u-u_h\|$.
The approximation orders are due to the properties of $\Pi^1_\mesh$
and $\Pi^\dDiv_\mesh$, cf.~Lemma~\ref{la_PidDiv}.
\end{proof}

\subsection{Nonlinear problem: Uzawa-type algorithm} \label{sec_disc_nlin}

We present a discretization of the nonlinear variational formulation \eqref{Mphiu}
and propose an Uzawa-type algorithm for its solution.

For ease of presentation we consider homogeneous essential boundary conditions for
$\bM$, i.e., $\bG=0$, and recall that $\eta\in H^1(\Omega)$ represents the Dirichlet
condition for $\phi$.
As in the linear case we use the discrete space $\X_N(\mesh)$ to approximate
$\bM$ and define an approximation of $u$ by postprocessing.
Function $\phi$ is approximated by standard continuous piecewise polynomials.

A discretization of \eqref{Mphiu} then reads:
\emph{Find $(\bM_h,\phi_h) \in \X_N(\mesh)\times \bigl(P^p(\mesh)\cap H^1(\Omega)\bigr)$ such that}
\begin{subequations}\label{Mphiudisc}
\begin{align}
   \label{Mphiu1disc}
   &\ip{\bM_h}{\dbM}_{\dDiv,\phi_h}
   = \frac{B}m \vdual{f}{\divDiv\dbM + q^2\bT(\phi_h)\colon\dbM}
   - B\dual{\trGg{}(g)}{\dbM},\\
   \label{Mphiu3disc}
   &K \vdual{\grad\phi_h}{\grad\dphi} + Bq^2\vdual{\bM_h\colon \bT'(\phi_h)u_h}{\dphi}
   = 0,\\
   \label{Mphiu4disc}
   &\phi_h|_\Gamma = \eta_h 
  \end{align}
\emph{for any $(\dbM,\dphi)\in \X_N(\mesh)\times \bigl(P^p(\mesh)\cap H^1_0(\Omega)\bigr)$ where}
\begin{align} \label{Mphiu5disc}
  u_h = u_h(\bM_h,\phi_h)
      = \Pi_\mesh^1\bigl(\frac1m f-\frac{B}m(\divDiv\bM_h + q^2\bT(\phi_h)\colon\bM_h)\bigr).
\end{align}
\end{subequations}
Here, $\eta_h\in \bigl(P^p(\mesh)\cap H^1(\Omega)\bigr)|_\Gamma$
is a quasi-interpolation of $\eta|_\Gamma$.
In our numerical experiments we use the nodal interpolation or the $L^2(\Gamma)$ projection.
As we expect the convergence $\|\phi-\phi_h\|_{H^1(\Omega)} = \cO(h^{p})$
for sufficiently smooth solution, we choose $p=2$ to be comparable with
the linear case.

Our algorithm solves \eqref{Mphiu1disc} for given $\phi_h$,
followed by an update of $\phi_h$ by solving~\eqref{Mphiu3disc}
for given $\bM_h$ from the previous step and $u_h$ defined by \eqref{Mphiu5disc}.
With the abbreviation
\[
   F_{\phi_h}(\dbM) := \frac{B}m \vdual{f}{\divDiv\dbM + q^2\bT(\phi_h)\colon\dbM}
\]
it reads as follows.

\begin{algorithm}[h]
	\SetKwData{Left}{left}
	\SetKwData{This}{this}
	\SetKwData{Up}{up}
	\SetKwFunction{Union}{Union}
	\SetKwFunction{FindCompress}{FindCompress}
	\SetKwInOut{Input}{Input}
	\SetKwInOut{Output}{Output}
  \Input{Parameters $\alpha>0$, $\tau_{\bM}>0$, $\tau_\phi>0$ and initial guess $\phi_h$ with $\phi_h|_\Gamma = \eta_h$}
  \Output{Approximations $u_h\approx u$, $\bM_h\approx \bM$, $\phi_h\approx \phi$}
	\BlankLine
  \Repeat{$\mathrm{res}_{\bM}<\tau_{\bM}$}{
      Solve $\bM_h\in\X_N(\mesh)$: $\ip{\bM_h}{\dbM}_{\dDiv,\phi_h}
      = F_{\phi_h}(\dbM)$ $\forall\dbM\in \X_N(\mesh)$\;
      Compute $u_h := \Pi_h^1\Big(\frac1m f-\frac{B}m(\divDiv\bM_h + q^2\bT(\phi_h)\colon\bM_h)\Big)$\;
      \Repeat{$\mathrm{res}_{\phi}<\tau_{\phi}$}{
        Solve $w\in H^1_0(\Omega)\cap P^p(\mesh)$:
        $K \vdual{\nabla w}{\nabla\dphi}
        = K \vdual{\nabla \phi_h}{\nabla\dphi}+Bq^2\vdual{\bM_h\colon\bT'(\phi_h)u_h}{\dphi}$
        $\forall\dphi\in H^1_0(\Omega)\cap P^p(\mesh)$\;
        Update $\phi_h \leftarrow \phi_h -\alpha w$\;
        Compute residual $\mathrm{res}_\phi = \sqrt{K} \|\nabla w\|$\;
      }
      Solve $\bN\in \X_N(\mesh)$:
      $\ip{\bN}{\dbM}_{\dDiv,\phi_h}
      = \ip{\bM_h}{\dbM}_{\dDiv,\phi_h}-F_{\phi_h}(\dbM)$ $\forall\dbM\in \X_N(\mesh)$\;
      Compute residual $\mathrm{res}_{\bM} = \norm{\bN}_{\dDiv,\phi_h}$\;
    }
	\caption{Uzawa-type algorithm for solving~\eqref{Mphiudisc}.}\label{alg:Uzawa}
\end{algorithm}

\input{sec_numeric}

\FloatBarrier

\bibliographystyle{siam}

\input{paper.bbl}
\end{document}

%% file: header.tex
\newtheorem{theorem}{Theorem}
\newtheorem{lemma}[theorem]{Lemma}

\newtheorem{remark}[theorem]{Remark}
\theoremstyle{definition}

\newcommand{\<}{\langle{}}
\renewcommand{\>}{\rangle}
\newcommand{\wtilde}{\widetilde}

\newcommand{\R}{\ensuremath{\mathbb{R}}}
\newcommand{\ip}[2]{\llangle#1\hspace*{.5mm},#2\rrangle}
\newcommand{\dual}[2]{\<#1\hspace*{.5mm},#2\>}
\newcommand{\vdual}[2]{(#1\hspace*{.5mm},#2)}
\newcommand{\set}[2]{\left\{#1;\; #2\right\}}

\newcommand{\bn}{{\boldsymbol{n}}}
\newcommand{\bt}{{\boldsymbol{t}}}
\newcommand{\pit}{{\boldsymbol{\pi}_t}}

\DeclareMathOperator{\diam}{diam}
\DeclareMathOperator{\sym}{sym}
\DeclareMathOperator*{\argmin}{arg min}
\let\hom\undefined
\def\hom{\mathrm{hom}}

\newcommand{\Grad}{\boldsymbol{\nabla}}
\newcommand{\grad}{\nabla}
\newcommand{\Gradgrad}{\Grad\grad}
\newcommand{\Ggrad}{\mathrm{Ggrad}}
\newcommand{\dDiv}{\mathrm{dDiv}}
\let\div\undefined
\DeclareMathOperator{\div}{div}
\DeclareMathOperator{\Div}{Div}
\DeclareMathOperator{\divDiv}{divDiv}
\DeclareMathOperator{\nDiv}{nDiv_{eff}}

\newcommand{\bx}{\boldsymbol{x}}
\newcommand{\bnu}{\boldsymbol{\nu}}
\newcommand{\bphi}{\boldsymbol{\phi}}
\newcommand{\bpsi}{\boldsymbol{\psi}}
\newcommand{\bT}{\boldsymbol{T}}
\newcommand{\bA}{\boldsymbol{A}}
\newcommand{\bM}{\boldsymbol{M}}
\newcommand{\bQ}{\boldsymbol{Q}}

\newcommand{\bG}{\boldsymbol{G}}
\newcommand{\dbM}{\boldsymbol{\delta\!M}}
\newcommand{\dphi}{\delta\!\phi}
\newcommand{\dpsi}{\delta\!\psi}

\newcommand{\du}{\delta\!u}

\newcommand{\cE}{\mathcal{E}}
\newcommand{\cF}{\mathcal{F}}

\newcommand{\cO}{\mathcal{O}}
\let\SS\undefined
\newcommand{\SS}{\mathbb{S}}

\newcommand{\trGg}[1]{\mathrm{tr}_{#1}^{\mathrm{Ggrad}}} 
\newcommand{\trdD}[1]{\mathrm{tr}_{#1}^{\mathrm{dDiv}}} 

\newcommand{\G}[1]{{\Gamma_{\mathit{#1}}}}
\newcommand{\JD}{{J_D}}
\newcommand{\JN}{{J_N}}

\newcommand{\mesh}{\mathcal{T}}
\newcommand{\el}{T}
\newcommand{\RT}{\mathcal{RT}}
\newcommand{\X}{\mathbb{X}}

%% file: sec_numeric.tex

\def\err{\mathrm{err}}
\def\errMLtwo{\err_{L^2(\Omega)}}
\def\errDDivM{\err_{\dDiv}}

\section{Numerical experiments}\label{sec_numeric}

\begin{minipage}{0.7\textwidth}
We present several numerical experiments that illustrate the performance of
our methods for linear and nonlinear cases in two space dimensions.
Throughout, we select $m=1$, $\Omega=(0,1)^2$, and use the initial triangulation
depicted on the right.

Starting from the initial triangulation we generate a sequence of quasi-uniform regular meshes
by dividing each triangle into four triangles of the same area using the (iterated)
newest vertex bisection algorithm. We note that for all meshes $\mesh$, $\diam(\el)$ and
$(\#\mesh)^{-1/2}$ are proportional to a mesh parameter $h$ for all $\el\in\mesh$.\\
\end{minipage}
\qquad
\begin{minipage}{0.3\textwidth}
\scalebox{0.45}{\input{fig_mesh}}
\end{minipage}


In cases where the exact solution is known we provide results for the individual terms of
the energy norm (with $m=1$), 
\begin{align*}
  \errMLtwo(\bM_h) &:= \sqrt{B}\|\bM-\bM_h\|, \\
  \errDDivM(\bM_h) &:= B\|\divDiv(\bM-\bM_h)+q^2\bT\colon(\bM-\bM_h)\|
\end{align*}
as well as the error $\|u-u_h\|$ of the postprocessed density variation $u_h$.

\subsection{Linear problem with manufactured solution}\label{sec_numeric_linKnown}
We consider the manufactured functions
\begin{align*}
  \bnu(x,y) &= \begin{pmatrix}
    \cos(\tfrac\pi2(y-\tfrac12)) \\
    \sin(\tfrac\pi2(y-\tfrac12))
  \end{pmatrix}, \quad
  u(x,y) = \sin(q[(x,y)\cdot\bnu(x,y)]), \quad
  \bT(x,y) = \bnu(x,y)\bnu(x,y)^\top
\end{align*}
for all $(x,y)\in \Omega$. Given $q>0$ we set $B = 1/q^4$ and
compute load $f$ through~\eqref{eq_u} where $\bM$ is defined as in~\eqref{M}. 
We prescribe the values of $u$ and $\nabla u$ on $\Gamma_{hc}:=\Gamma$
(fully hard-clamped case in~\eqref{bc}).
Figure~\ref{fig:ExampleLinErrors} shows the error contributions in the energy norm and
the $L^2(\Omega)$ error of the postprocessed solution for $q=1,20,40,60$.
We observe asymptotically optimal rates $\cO(h^2)$, though larger values of $q$
cause longer pre-asymptotic phases. This is particularly evident for $q=40$ and $q=60$.
The convergence behavior complies with the observations of Theorem~\ref{thm_disc_lin}.

\begin{figure}
  \begin{center}
    \includegraphics[width=\textwidth]{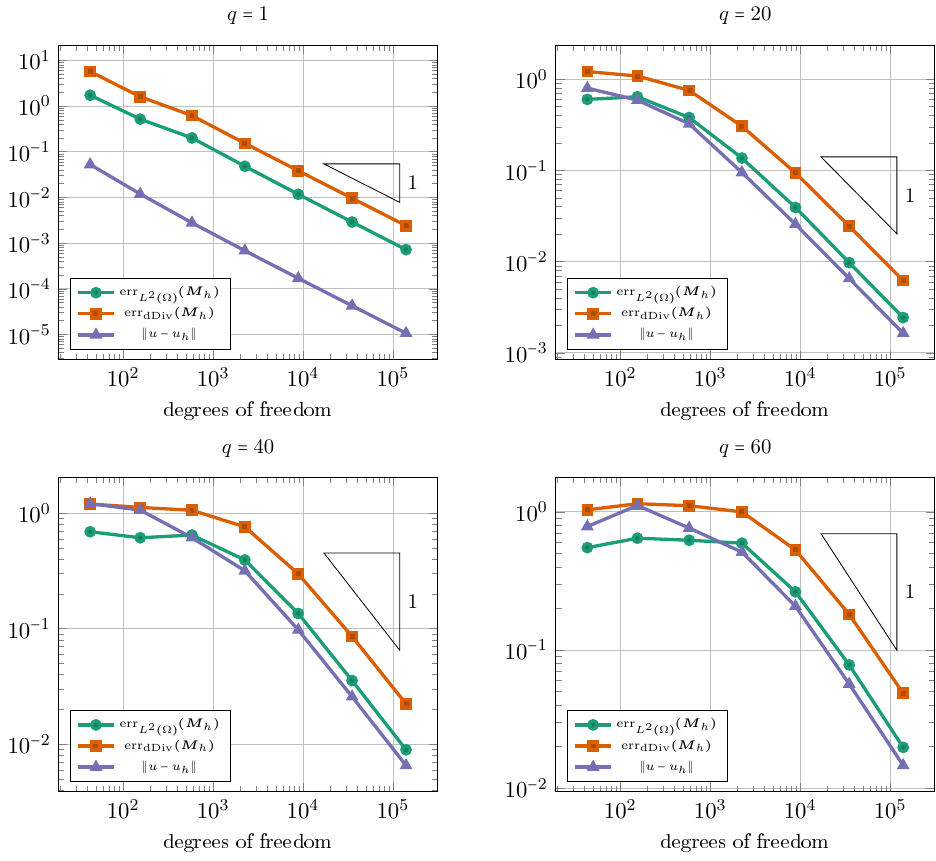}
  \end{center}
  \caption{Error contributions for the linear problem from Section~\ref{sec_numeric_linKnown} for $q=1,20,40,60$.}
  \label{fig:ExampleLinErrors}
\end{figure}

\subsection{Linear problem with unknown solution}\label{sec_numeric_linUnknown}
We choose $f(x,y) = 1$ for $(x,y)\in\Omega$ and
set $\bT = \bnu\bnu^\top$ for three different choices of $\bnu$:
\begin{align*}
  \bnu^{(1)}(x,y) &= \begin{pmatrix}
    \cos(\tfrac\pi2(y-\tfrac12)) \\
    \sin(\tfrac\pi2(y-\tfrac12))
  \end{pmatrix}, \\
  \bnu^{(2)}(x,y) &= 
  \begin{cases}
    \begin{pmatrix}
      \cos(\tfrac\pi2(y-\tfrac12)) \\
      \sin(\tfrac\pi2(y-\tfrac12))
    \end{pmatrix}
    & x>\tfrac12, \\
    \begin{pmatrix}
      0 \\
      -1
    \end{pmatrix}
    & x<\tfrac12, y<\tfrac12, \\
    \begin{pmatrix}
      0 \\
      1
    \end{pmatrix}
    & x<\tfrac12, y>\tfrac12,
  \end{cases}
  \\
  \bnu^{(3)}(x,y) &= \begin{pmatrix}
    \frac{y-\tfrac12}{|(x-\tfrac14,y-\tfrac12)|^2} - \frac{y-\tfrac12}{|(x-\tfrac34,y-\tfrac12)|^2}\\
    \frac{x-\tfrac34}{|(x-\tfrac34,y-\tfrac12)|^2} - \frac{x-\tfrac14}{|(x-\tfrac14,y-\tfrac12)|^2} 
  \end{pmatrix}
    C(x,y).
\end{align*}
Here, $C(x,y)$ is chosen such that $|\bnu^{(3)}|=1$ a.e. in $\Omega$.
Note that $\bnu^{(1)}$ corresponds to the vector field considered in Section~\ref{sec_numeric_linKnown} and is continuous whereas $\bnu^{(2)}$ and $\bnu^{(3)}$ are discontinuous vector fields.
Field $\bnu^{(3)}$ represents a rotated and normalized electric dipole. 
The three fields are visualized in the left columnn of Figure~\ref{fig:ExampleLinVD}.

For the computations we choose $B = 10^{-5}$ and $q=40$. This choice of parameters is motivated by the setup presented in~\cite[Fig.~6]{PevnyiSS_14_MSL}.
Furthermore, we use free boundary conditions for $\bM$,
i.e., $\Gamma_f := \Gamma$ and $\bG=0$ in~\eqref{bc}.
The right column of Figure~\ref{fig:ExampleLinVD} visualizes the postprocessed solution $u_h$ on a mesh with $16384$ elements for the three fields $\bnu^{(j)}$, $j=1,2,3$ (top to bottom). 
Note that these plots correspond to the density variation in liquid crystal simulations.
Comparing the plots in each row we see that the waves seem to align with the director field
as is observed in other liquid crystal simulations with similar energy functional,
cf.~\cite{PevnyiSS_14_MSL}.

\begin{figure}
  \begin{center}
    \includegraphics[height=0.9\textheight]{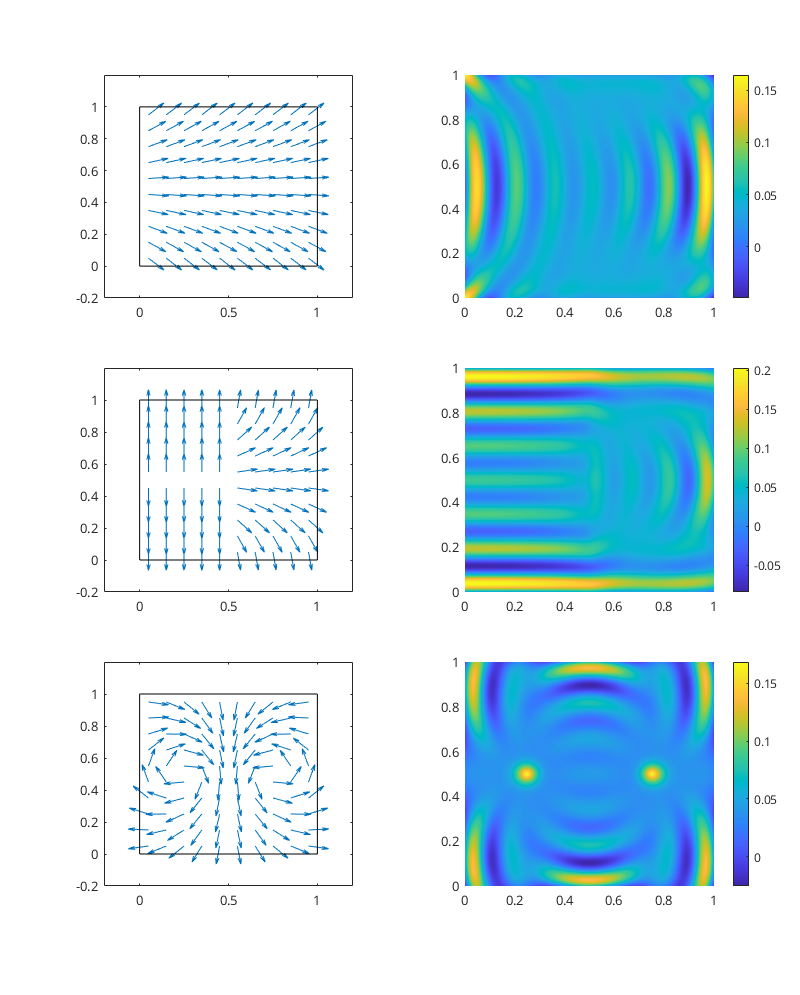}
  \end{center}
  \caption{Vectors $\bnu^{(1)}$, $\bnu^{(2)}$, $\bnu^{(3)}$ (left column, top to bottom) with corresponding approximations $u_h$ (right column) for the linear problem from Section~\ref{sec_numeric_linUnknown}.}\label{fig:ExampleLinVD}
\end{figure}

\begin{figure}
  \begin{center}
    \includegraphics[width=0.5\textwidth]{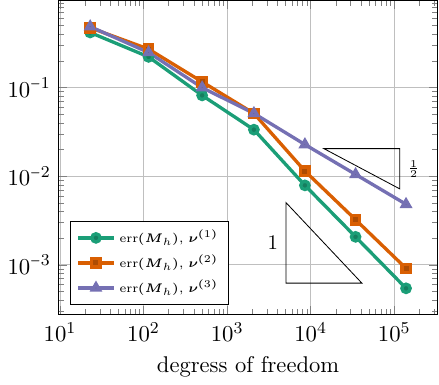}
  \end{center}
  \caption{Experimental orders of convergence for the setup of Section~\ref{sec_numeric_linUnknown} using the vector fields $\bnu^{(j)}$, $j=1,2,3$.}
  \label{fig:ExampleLinUnknownErrors}
\end{figure}

Since the exact solution is not known we cannot report exact errors $\norm{\bM-\bM_h}_{\dDiv}$.
Instead, we exploit Galerkin orthogonality
\begin{align*}
  \norm{\bM-\bM_h}_{\dDiv}^2 = \norm{\bM}_{\dDiv}^2-\norm{\bM_h}_{\dDiv}^2
\end{align*}
and estimate $\norm{\bM}_{\dDiv}$ by extrapolating the values of $\norm{\bM_h}_{\dDiv}$ using Aitken's $\Delta^2$ method.
Let us call this value $E^\star$.
Figure~\ref{fig:ExampleLinUnknownErrors} shows the experimental errors
\begin{align*}
  \err(\bM_h):= \sqrt{(E^\star)^2-\norm{\bM_h}_\dDiv^2}.
\end{align*}
For vector fields $\bnu^{(1)}$ and $\bnu^{(2)}$ we obtain optimal convergence rates
$\err(\bM_h)=\cO(h^2)$ whereas for $\bnu^{(3)}$, convergence is reduced to $\cO(h)$.
This indicates an exact solution with reduced regularity.

\subsection{Nonlinear problem with known solution}
We consider the manufactured solution
\begin{align*}
  \phi(x,y) &= -\frac\pi4 + \frac\pi2 y^3, \quad
  u(x,y) = \sin(q(x-\tfrac12,y-\tfrac12)^\top\cdot\bnu) \quad \text{for }(x,y)\in \Omega
\end{align*}
with $\bnu = (\cos(\phi),\sin(\phi))^\top$.
Given $q>0$ we set $B = 1/q^4$, choose $K=1$,
and compute load $f$ through~\eqref{eq_u} where $\bM$ is defined as in~\eqref{M}. 
We prescribe the values of $u$ and $\nabla u$ on $\Gamma_{hc}:=\Gamma$
(fully hard-clamped case in~\eqref{bc}).
Since $\phi$ does not satisfy $-\Delta \phi + Bq^2\bT\colon \bM u = 0$
we introduce the additional source term
$f_\phi := -\Delta \phi + Bq^2\bT\colon \bM u$ in~\eqref{Mphiu3}
and adapt the discrete scheme accordingly.

We run Algorithm~\ref{alg:Uzawa} with $\alpha = 0.5$ and $\tau_{\bM} = \tau_\phi = 10^{-6}$.

\begin{figure}
  \begin{center}
    \includegraphics[width=\textwidth]{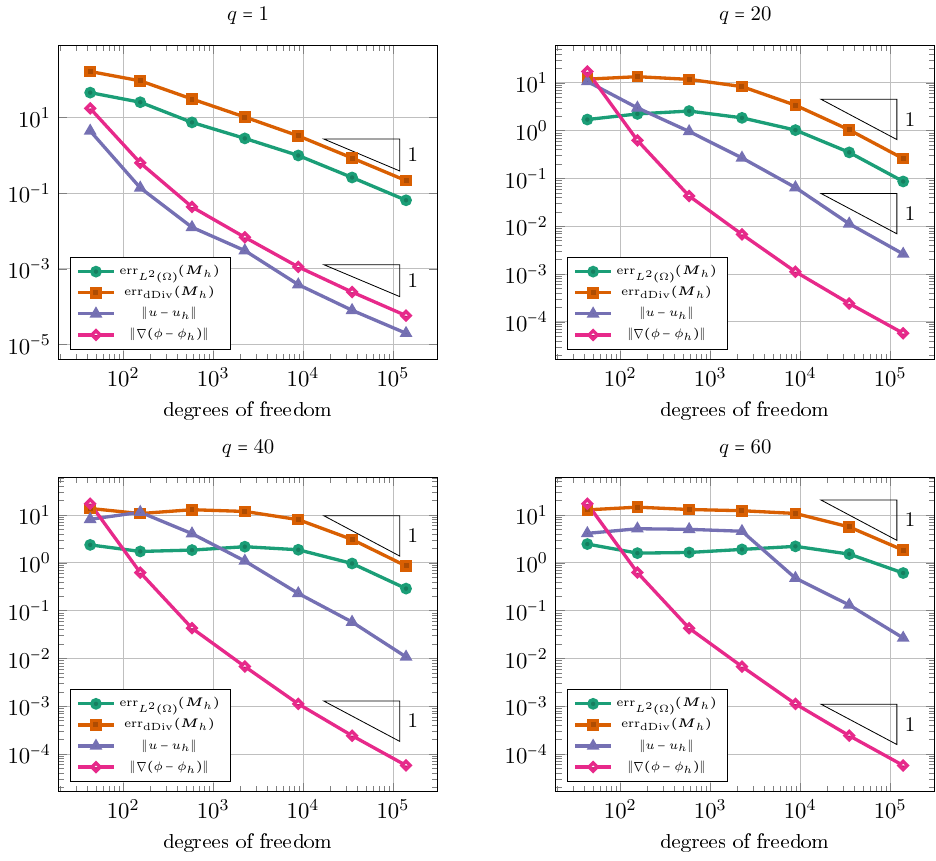}
  \end{center}
  \caption{Error contributions for the nonlinear problem with known solution for $q=1,20,40,60$.}
  \label{fig:ExampleNonLinErrors}
\end{figure}

Figure~\ref{fig:ExampleNonLinErrors} displays the error contributions for different values of $q$. 
We observe optimal convergence orders but, as in the linear case, larger values of $q$ lead to a longer preasymptotic phase. 
This is particularly observed for $q=40$ and $q = 60$.
Table~\ref{tab:nonLinExact} lists the iterations numbers for the exterior and interior loops
of the Uzawa-algorithm~\ref{alg:Uzawa}. We observe a decrease of inner iterations during
the outer loops, whereas the table presents their averages. In the case of $q=1$ and the two
final meshes, the outer loop converged to a residual of about $2\cdot 10^{-6}$ within
$25$ iterations which did not decrease afterwards.
We attribute this to the stronger nonlinear perturbations for smaller $q$, cf.~\eqref{Mphiu}.
Note that the relevant factor is $Bq^2=q^{-2}$ in this example.
For more practical values of $q$, the iteration
numbers are generally moderate with little variation.

\begin{table}[htbp]
  \begin{center}
  {\small
  \begin{tabular}{|c|cc|cc|cc|cc|}
    \hline
    \multicolumn{1}{|c|}{} &\multicolumn{2}{|c|}{$q=1$} &\multicolumn{2}{|c|}{$q=20$} &\multicolumn{2}{|c|}{$q=40$} &\multicolumn{2}{|c|}{$q=60$}\\
     \cline{2-9}
$\#\mesh$ &  outer & inner & outer & inner & outer & inner & outer & inner \\\hline
16    & 7  & 12.28 & 4 & 12.25 & 5 & 12.00 & 6 & 11.00 \\\hline
64    & 5  & 10.40 & 6 & 9.83  & 4 & 12.00 & 4 & 12.50 \\\hline
256   & 3  & 10.66 & 5 & 9.00  & 5 & 9.00  & 4 & 10.00 \\\hline
1024  & 3  & 9.33  & 4 & 8.50  & 4 & 9.50  & 4 & 10.50 \\\hline
4096  & 25 & 1.84  & 3 & 9.66  & 3 & 9.33  & 4 & 8.25 \\\hline
16384 & 25 & 1.72  & 3 & 8.00  & 3 & 8.33  & 3 & 8.66 \\\hline
  \end{tabular}
}
\end{center}
\caption{Outer and inner iterations in the Uzawa-type algorithm for the problem with known solution.
The columns labeled \emph{inner} show the average number of inner iterations.
For $q=1$ and $\#\mesh=4096$ resp. $\#\mesh=16384$ the outer loop did not converge to the prescribed tolerance within $25$ outer.}
  \label{tab:nonLinExact}
\end{table}

\subsection{Nonlinear problem with unknown solution}\label{sec_numeric_nonlinUnknown}
As in the previous section we select $\alpha=0.5$ and $\tau_{\bM} = \tau_\phi = 10^{-6}$
in Algorithm~\ref{alg:Uzawa}.
We choose $q=40$, $B = 10^{-5}$, $K=1$, $\bnu = (\cos(\phi),\sin(\phi))^\top$, and
$f(x,y)=1$ for $(x,y)\in\Omega$. We select homogeneous free boundary conditions for $\bM$
everywhere, $\Gamma_f :=\Gamma$ and $\bG:=0$ in \eqref{bc}, and
consider three Dirichlet choices for $\phi$: $\eta = \eta_j$, $j=1,2,3$, with
\begin{align*}
  \eta_1(x,y) &:= -\frac\pi4+\frac\pi2 y, \\
  \eta_2(x,y) &:= \frac\pi2\sin(2\pi(y-1/2)), \\
  \eta_3(x,y) &:= \begin{cases}
    -\tfrac\pi2+\pi y & x>\tfrac12, \\
    \tfrac\pi2 & x<\tfrac12, y>\tfrac12, \\
    -\tfrac\pi2 & x<\tfrac12, y<\tfrac12.
    \end{cases}
\end{align*}
Note that $\eta_3$ has a finite jump at $(x,y) = (0,\tfrac12)$ and is therefore not
an $H^1(\Omega)$ trace.
For the computations we replace data $\eta_j|_\Gamma$ by their
$L^2(\Gamma)$-projections onto $\bigl(P^2(\mesh)\cap H^1(\Omega)\bigr)|_\Gamma$.

In these cases the exact solution is unknown.
Furthermore, the inner product $\ip\cdot\cdot_{\dDiv,\phi_h}$ depends
on the approximation $\phi_h$ and is expected to change in each step.
Therefore, we cannot use Galerkin orthogonality as in the linear case to estimate the error. 
Instead, we consider the discrete energy
\begin{align*}
  J(\bM_h,\phi_h;f) := \frac{B}2\|\bM_h\|^2 + \frac{m}2\|u_h\|^2 + \frac{K}2\|\nabla\phi_h\|^2 - \vdual{f}{u_h}
\end{align*}
and, again employ Aitken's $\Delta^2$ method to extrapolate it to a value $J^\star$. 
We then report on the error quantity
\begin{align*}
  \err(\bM_h,\phi_h) := \sqrt{|J^\star-J(\bM_h,\phi_h;f)|}.
\end{align*}
Figure~\ref{fig:ExampleNonLinUnknownErrors} shows these errors for the Dirichlet boundary values $\eta_1$ and $\eta_2$.
We observe optimal convergence $\cO(h^2)$ in both cases. 
We do not show the energy error in the case of $\eta_3$. Recall that $\eta_3$ is discontinuous
on the boundary and one expects that $\|\nabla \phi_h\|\to \infty$ for $h\to 0$,
thus $J(\bM_h,\phi_h;f)\to \infty$ for $h\to 0$. 
In numerical experiments (not reported) we observe that the discrete energy
is indeed increasing on a sequence of quasi-uniform meshes.
The discontinuity of $\phi$ at $(0,1/2)$ is confirmed by the approximation $\phi_h$
on a mesh with $16384$ elements, see Figure~\ref{fig:discontSol}.

\begin{figure}[!htb]
  \begin{center}
    \includegraphics[width=0.5\textwidth]{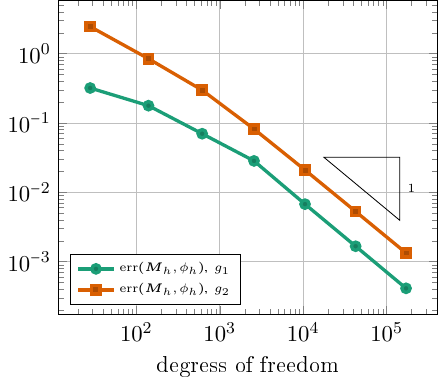}
  \end{center}
  \caption{Experimental orders of convergence for the setup of Section~\ref{sec_numeric_nonlinUnknown} using the Dirichlet data $\eta_j$, $j=1,2$.}
  \label{fig:ExampleNonLinUnknownErrors}
\end{figure}

\begin{figure}[!htb]
  \begin{center}
    \includegraphics[height=0.3\textheight]{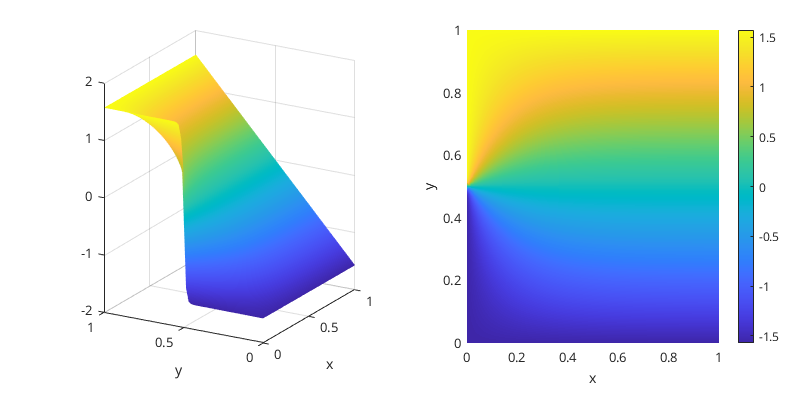}
  \end{center}
  \caption{Different views of $\phi_h$ for Dirichlet boundary values $\eta_3$ for the nonlinear problem with unknown solution.}\label{fig:discontSol}
\end{figure}

Finally, in Figure~\ref{fig:ExampleNonLinVD} we present the vectors $\bnu$ and the approximations of the density variation $u_h$ for the three different choices of boundary values $\eta_1,\eta_2,\eta_3$. 

\begin{figure}[!htb]
  \begin{center}
    \includegraphics[height=0.9\textheight]{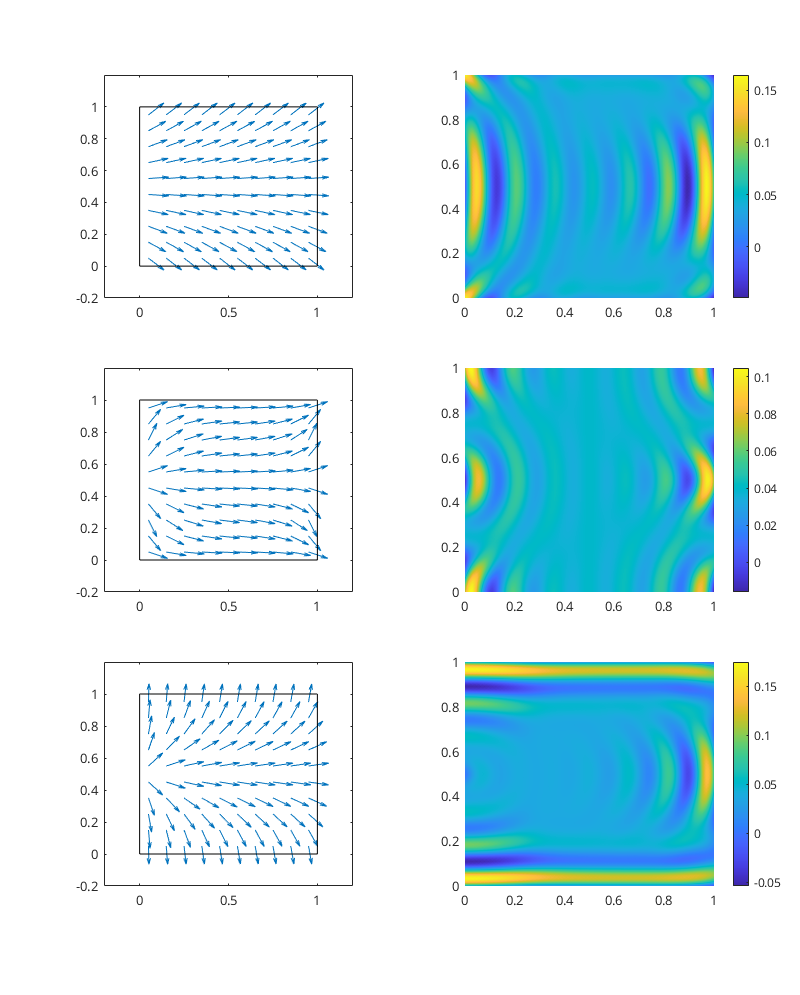}
  \end{center}
  \caption{Vectors and densities for the three different choices $\eta_1,\eta_2,\eta_3$ for the nonlinear problem with unknown solution.}\label{fig:ExampleNonLinVD}
\end{figure}

%% file: fig_mesh.tex
\setlength{\unitlength}{4144sp}%
\begingroup\makeatletter\ifx\SetFigFont\undefined%
\gdef\SetFigFont#1#2#3#4#5{%
  \reset@font\fontsize{#1}{#2pt}%
  \fontfamily{#3}\fontseries{#4}\fontshape{#5}%
  \selectfont}%
\fi\endgroup%
\begin{picture}(3666,3666)(4468,-7294)
\thicklines
{\color[rgb]{0,0,0}\put(4501,-7261){\framebox(3600,3600){}}
}%
{\color[rgb]{0,0,0}\put(4501,-7261){\line( 1, 1){3600}}
}%
{\color[rgb]{0,0,0}\put(4501,-3661){\line( 1,-1){3600}}
}%
\end{picture}%

%% file: paper.bbl
\begin{thebibliography}{10}

\bibitem{Adams}
{\sc R.~A. Adams}, {\em Sobolev spaces}, Academic Press, New York-London, 1975.
\newblock Pure and Applied Mathematics, Vol. 65.

\bibitem{AdlerAEML_15_EMF}
{\sc J.~H. Adler, T.~J. Atherton, D.~B. Emerson, and S.~P. MacLachlan}, {\em An
  energy-minimization finite-element approach for the {F}rank-{O}seen model of
  nematic liquid crystals}, SIAM J. Numer. Anal., 53 (2015), pp.~2226--2254.

\bibitem{BallB_15_DOP}
{\sc J.~M. Ball and S.~J. Bedford}, {\em Discontinuous order parameters in
  liquid crystal theories}, Molecular Crystals and Liquid Crystals, 612 (2015),
  pp.~1--23.

\bibitem{BorthagarayNW_20_SPF}
{\sc J.~P. Borthagaray, R.~H. Nochetto, and S.~W. Walker}, {\em A
  structure-preserving {FEM} for the uniaxially constrained {Q}-tensor model of
  nematic liquid crystals}, Numer. Math., 145 (2020), pp.~837--881.

\bibitem{ChenH_25_NDD}
{\sc L.~Chen and X.~Huang}, {\em A new div-div-conforming symmetric tensor
  finite element space with applications to the biharmonic equation}, Math.
  Comp., 94 (2025), pp.~33--72.

\bibitem{ChenYZ_17_SOL}
{\sc R.~Chen, X.~Yang, and H.~Zhang}, {\em Second order, linear, and
  unconditionally energy stable schemes for a hydrodynamic model of smectic-{A}
  liquid crystals}, SIAM J. Sci. Comput., 39 (2017), pp.~A2808--A2833.

\bibitem{FarrellHM_FED}
{\sc P.~E. Farrell, A.~Hamdan, and S.~P. MacLachlan}, {\em Finite-element
  discretization of the smectic density equation}, {arXiv}:2207.12916, 2022.

\bibitem{FuehrerH_25_MKL}
{\sc T.~F{\"u}hrer and N.~Heuer}, {\em Mixed finite elements for
  {Kirchhoff}--{Love} plate bending}, Math. Comp., 94 (2025), pp.~1065--1099.

\bibitem{FuehrerHN_19_UFK}
{\sc T.~F{\"u}hrer, N.~Heuer, and A.~H. Niemi}, {\em An ultraweak formulation
  of the {Kirchhoff}--{Love} plate bending model and {DPG} approximation},
  Math. Comp., 88 (2019), pp.~1587--1619.

\bibitem{FuehrerHS_20_UFR}
{\sc T.~F{\"u}hrer, N.~Heuer, and F.-J. Sayas}, {\em An ultraweak formulation
  of the {Reissner}--{Mindlin} plate bending model and {DPG} approximation},
  Numer. Math., 145 (2020), pp.~313--344.

\bibitem{DeGennesP_93}
{\sc P.~G.~D. Gennes and J.~Prost}, {\em The Physics of Liquid Crystals},
  Oxford University Press, 1993.

\bibitem{GuillenGT_15_ASA}
{\sc F.~Guill\'en-Gonz\'alez and G.~Tierra}, {\em Approximation of smectic-{A}
  liquid crystals}, Comput. Methods Appl. Mech. Engrg., 290 (2015),
  pp.~342--361.

\bibitem{GuillenGGS_13_LMF}
{\sc F.~M. Guill\'en-Gonz\'alez and J.~V. Guti\'errez-Santacreu}, {\em A linear
  mixed finite element scheme for a nematic {E}ricksen-{L}eslie liquid crystal
  model}, ESAIM Math. Model. Numer. Anal., 47 (2013), pp.~1433--1464.

\bibitem{LiuW_00_ALC}
{\sc C.~Liu and N.~J. Walkington}, {\em Approximation of liquid crystal flows},
  SIAM J. Numer. Anal., 37 (2000), pp.~725--741.

\bibitem{LiuW_02_MMA}
\leavevmode\vrule height 2pt depth -1.6pt width 23pt, {\em Mixed methods for
  the approximation of liquid crystal flows}, M2AN Math. Model. Numer. Anal.,
  36 (2002), pp.~205--222.

\bibitem{MaityMN_21_DGF}
{\sc R.~R. Maity, A.~Majumdar, and N.~Nataraj}, {\em Discontinuous {G}alerkin
  finite element methods for the {L}andau--de {G}ennes minimization problem of
  liquid crystals}, IMA J. Numer. Anal., 41 (2021), pp.~1130--1163.

\bibitem{MaityMN_23_PPE}
\leavevmode\vrule height 2pt depth -1.6pt width 23pt, {\em {\it {A} priori} and
  {\it a posteriori} error analysis for semilinear problems in liquid
  crystals}, ESAIM Math. Model. Numer. Anal., 57 (2023), pp.~3201--3250.

\bibitem{NochettoRY_22_GCP}
{\sc R.~H. Nochetto, M.~Ruggeri, and S.~Yang}, {\em Gamma-convergent
  projection-free finite element methods for nematic liquid crystals: the
  {E}ricksen model}, SIAM J. Numer. Anal., 60 (2022), pp.~856--887.

\bibitem{PevnyiSS_14_MSL}
{\sc M.~Y. Pevnyi, J.~V. Selinger, and T.~J. Sluckin}, {\em Modeling smectic
  layers in confined geometries: Order parameter and defects}, Phys. Rev. E, 90
  (2014), p.~032507.

\bibitem{Rindler_18_CV}
{\sc F.~Rindler}, {\em Calculus of variations}, Universitext, Springer, Cham,
  2018.

\bibitem{WangZZ_21_MCL}
{\sc W.~Wang, L.~Zhang, and P.~Zhang}, {\em Modelling and computation of liquid
  crystals}, Acta Numer., 30 (2021), pp.~765--851.

\bibitem{XiaF_23_VNA}
{\sc J.~Xia and P.~E. Farrell}, {\em Variational and numerical analysis of a
  {\bf {q}}-tensor model for smectic-{A} liquid crystals}, ESAIM Math. Model.
  Numer. Anal., 57 (2023), pp.~693--716.

\bibitem{XiaMAF_21_SLG}
{\sc J.~Xia, S.~MacLachlan, T.~J. Atherton, and P.~E. Farrell}, {\em Structural
  landscapes in geometrically frustrated smectics}, Phys. Rev. Lett., 126
  (2021), p.~177801.

\end{thebibliography}
